\documentclass[a4paper,10pt]{amsart}

\usepackage[english]{babel}
\usepackage{dsfont} 

\usepackage{amsmath}
\usepackage{amsfonts}
\usepackage{amssymb} 
\usepackage{amsthm} 
\usepackage{color}

\newtheorem{thm}{Theorem}
\newtheorem{dfn}[thm]{Definition}
\newtheorem{lem}[thm]{Lemma}

\newtheorem{exa}[thm]{Example}
\newtheorem{prop}[thm]{Proposition}
\newtheorem{rem}[thm]{Remark}
\newtheorem{cor}[thm]{Corollary}

\newcommand{\hide}[1]{}
\newcommand{\sN}{\mathsf{N}}
\newcommand{\sF}{\mathsf{F}}
\newcommand{\sA}{\mathsf{A}}
\newcommand{\sB}{\mathsf{B}}

\newcommand{\R}{\mathds{R}}
\newcommand{\N}{\mathds{N}}

\newcommand{\cF}{\mathcal{F}}
\newcommand{\cL}{\mathcal{L}}

\newcommand{\cN}{\mathcal{N}}
\newcommand{\cM}{\mathcal{M}}

\newcommand{\cK}{\mathcal{K}}

\newcommand{\cU}{\mathcal{U}}
\newcommand{\cS}{\mathcal{S}}

\newcommand{\cV}{\mathcal{V}}
\newcommand{\cW}{\mathcal{W}}
\newcommand{\cX}{\mathcal{X}}
\newcommand{\cZ}{\mathcal{Z}}

\newcommand{\gB}{\mathfrak{B}}

\newcommand{\ov}{\overline}

\newcommand{\cA}{\mathcal{A}}

\newcommand{\Lin}{\mathrm{Lin}}

\newcommand{\nset}{\ensuremath{\mathbb{N}}}

\newcommand{\Zset}{\ensuremath{\mathcal{Z}}}

\newcommand{\pset}{\ensuremath{\mathbb{P}}}
\newcommand{\rset}{\ensuremath{\mathbb{R}}}

\newcommand{\supp}{\ensuremath{\mathrm{supp}\,}}

\newcommand{\rank}{\ensuremath{\mathrm{rank}\,}}

\newcommand{\range}{\ensuremath{\mathrm{range}\,}}

\newcommand{\codim}{\ensuremath{\mathrm{codim}\,}}

\author{Philipp J.\ di~Dio}

\author{Konrad Schm\"udgen}
\address{Universit\"at Leipzig, Mathematisches Institut, Augustusplatz 10/11, D-04109 Leipzig, Germany}
\address{Max Planck Institute for Mathematics in the Sciences, Inselstra{\ss}e 22, D-04103 Leipzig, Germany}

\begin{title}
{The multidimensional truncated Moment Problem:  Atoms, Determinacy, and Core Variety}
\end{title}
\date{}
\begin{document}

\begin{abstract}
This paper is about  the  moment problem on a finite-dimensional vector space of continuous functions. We investigate the  structure of the convex cone of moment functionals (supporting hyperplanes, exposed faces, inner points) and treat various important special topics on moment functionals (determinacy, set of atoms of representing measures, core variety).
\end{abstract}

\maketitle


\textbf{Key  words:} truncated moment problem, moment cone, convex cone

\section{Introduction}

Let ${\sf N}$ be a finite subset of $\N_0^n$,  $n\in \N$, and  ${\sA}=\{x^\alpha: \alpha \in \sN\}$, $\cA = \Lin\,\sA$ the  span of  associated monomials, where $x^\alpha=x_1^{\alpha_1}\cdots x_n^{\alpha_n}$,  $\alpha=(\alpha_1,\dots,\alpha_n)\in \N_0^n$.  Suppose that $\cK$ is a closed subset of $\R^n$. Let $s=(s_{\alpha})_{\alpha \in \sN}$ be a real sequence and let $L_s$ denote  the corresponding Riesz functional  on  ${\cA}$  defined by  $L_s(x^\alpha)=s_\alpha$, $\alpha\in \sN$.

The  truncated moment problem asks: 
{\it When does there exist a (positive) Radon measure $\mu$  on $\cK$ such that $x^\alpha$ is $\mu$-integrable and}
\begin{equation}\label{momentsalpha}
s_\alpha =\int_{\R^n} x^\alpha \, d\mu\quad \text{ for all}\quad  \alpha \in \sN?
\end{equation} 
Clearly, (\ref{momentsalpha}) is equivalent to
\begin{equation}\label{Lrep}
L_s(f)=\int_\cK f(x)\,d\mu\quad {\rm for}\quad f\in {\cA}.
\end{equation}
The Richter--Tchakaloff theorem (Proposition \ref{hrichtertheorem}) implies that in the affirmative case there is always a finitely atomic measure $\mu$ satisfying (\ref{momentsalpha}) and (\ref{Lrep}).

The multidimensional truncated moment problem was first studied in the unpublished Thesis of J.\ Matzke \cite{matzke} and by R.\ Curto and L.\ Fialkow \cite{curto2}, \cite{curto3}, see \cite{laurent} for a nice survey. The one-dimensional case is treated in the monographs \cite{karlin}, \cite{kreinnudel}.

In the present paper we  consider  the truncated moment problem in a more general setting. That is, we study moment functionals  on a finite-dimensional  vector space $E$ of real-valued continuous functions on a locally compact topological Hausdorff space $\cX$. 
The bridge to the truncated $\cK$-moment problem for polynomials as formulated above is obtained by letting  $E$ the vector space of restrictions $f\lceil \cK$  of functions $f\in {\cA}$ to   $\cX:=\cK$. In this manner the results of this paper give new results  concerning the truncated $\cK$-moment problem for polynomials.

Let us briefly describe the structure and the contents of this paper. In Section
\ref{prelimin} we recall basic notation, definitions and facts on moment sequences and moment functionals. Let $L$ be a moment functional on $E$.
The set $\cW(L)$ of possible atoms of  representing measures of  $L$ is investigated in Section \ref{setofatoms}. In Section \ref{Edeterminacy}, we characterize the determinacy of $L$  in terms of the set $\cW(L)$ (Theorem \ref{strictposLnonunique}). Three other important notions associated with   $L$ are studied in Sections
\ref{exposedfacesmomentcone} and \ref{corersetofatoms}. These are the cone $\cN_+(L)$ of nonnegative functions of $E$ which are annihiliated by $L$, the zero set $\cV_+(L)$ of $\cN_+(L)$  and the core variety $\cV(L)$ introduced by L. Fialkow \cite{fialkoCoreVari}. It is easily seen that  $\cW(L) \subseteq \cV_+(L)$. Equality holds if and only if the moment sequence of $L$ lies in the relative interior of an exposed face of the moment cone (Theorem \ref{thm:WV+cases}). It is proved that  the set $\cW(L)$ is equal to the core variety $\cV(L)$ (Theorem \ref{wlequadlvl}). In the last Section 
\ref{differentailstructure} we assume that $\cX=\rset^n$ and  $E\subseteq C^1(\rset^n;\rset)$.  Then the total derivative of the moment map is used to analyze the structure of  the moment cone. A number of characterizations of inner points of the moment cone are given (Theorem \ref{thm:W(s)V(s)innermomentsequence}).

\section{Moment sequences and moment functionals}\label{prelimin}

Throughout this paper, we will suppose the following:
\begin{itemize}
\item $\cX$ is a \textbf{locally compact} topological Hausdorff space,

\item $E$ is a \textbf{finite-dimensional}  vector space of  real continuous functions  on $\cX$,

\item ${\sF}:=\{f_1,\dots,f_m\}$ is  a fixed \textbf{vector space basis} of $E$. 
\end{itemize}

For a real sequence $s=(s_j)_{j=1}^m$ the Riez functionals $L_s$ is the  linear functional $L_s$ on $E$ defined by $L_s(f_j)=s_j,j=1,\dots,m.$ This  one-to-one correspondence between real sequences and real linear functionals on $E$ is often used in what follows.

Let  $M_+(\cX)$ denote the set of   Radon measures on $\cX$. By a \textit{Radon measure} on $\cX$ we mean a  measure $\mu:\gB(\cX)\to [0,+\infty]$ on the Borel $\sigma$-algebra $\gB(\cX)$ such that 
\[
\mu(M)= {\rm sup }~\{\mu(K): K\subseteq M,~ K~\text{compact}\} \quad\text{for}~~~ M\in \gB(\cX).
\]
Note that in our terminology  Radon measures are always nonnegative!

For $\mu \in M_+(\cX),$ let $\cL^1(\cX,\mu)$ denote the  real-valued $\mu$-integrable Borel functions on $\cX$. For $x\in \cX$, let $\delta_x\in M_+(\cX)$ be defined by $\delta_x(M)=1$ if $x\in M$ and $\delta_x(M)=0$ if $x\notin M$. A  measure $\mu\in M_+(\cX)$ such that $|{\rm supp}\, \mu|=k$  is called  {\it $k$-atomic}; this means that there are $k$ pairwise different points $x_1,\dots,x_k$ of $\cX$ and positive numbers $c_1,\dots,c_k$ such that $\mu=\sum_{j=1}^k c_j\delta_{x_j}$. We  consider the zero measure as  $0$-atomic measure. 
For $f\in C(\cX;\R)$ we set $\cZ(f):=\{x\in \cX: f(x)=0\}.$

\begin{dfn}\label{momentfunct}
We  say that a real sequence $s=(s_j)_{j=1}^m$ is  a  \emph{moment sequence} and the linear functional $L_s$ is a \emph{moment functional} if there exists a measure $\mu \in M_+(\cX)$ such that $E\subseteq \cL^1(\cX,\mu)$ and
\[s_j=\int_{\cX} f_j(x)~ d\mu\quad \text{for}~~j=1,\dots,m,\]
or equivalently, 
\[L_s(f)=\int_{\cX} f(x)~ d\mu,\quad \text{for}\ f\in E.\]
Any such measure $\mu$ is called a \emph{representing measure} of $s$ resp. $L_s$. The set of  all representing measures of $s$ resp. $L_s$  is denoted by $\cM_s=\cM_{L_s}$.

The \emph{moment cone} $\cS$ is the set of all moment sequences. The set of all moment functionals is denoted by $\cL$.
\end{dfn}

Clearly, $\cS$ is a  cone in $\rset^m$ and  $\cL$ is a cone in the dual space of $E$. 
The  map $s\mapsto L_s$ is a bijection of $\cS$ to $\cL$.

Thus, we have a one-to-one correspondence between  moment sequences $s$ and moment functionals $L_s$.  At some places we prefer to work with moment sequences, while at others moment functionals are more convenient. Let us adopt the following notational convention: If we introduce  a set  depending on a general moment sequence $s$ (or  moment functional $L_s$), we will take  the same set for the moment functional $L_s$ (or moment sequence $s$). That is, for the sets introduced in what follows we define $\cN_+(s)=\cN_+(L_s),\cV_+(s)=\cV_+(L_s),\cW(s)=\cW(L_s),\cV(s)=\cV(L_s).$

\begin{rem}\label{reramrwell-defined}
Let us discuss briefly how the results on  moment functionals on $E$ apply to the truncated $\cK$-moment problem  on $\cA$ stated in the introduction. We set $\cX=\cK$ and consider the  subspace $E:={\cA} \lceil \cX$  of
$C(\cX;\R)$. Let $L$ be a linear functional on $\cA$. If
\begin{equation}\label{constsi}
L(f)=0\quad \text{for}~~~f\in {\cA}\ \text{with}\ f\lceil \cK=0,
\end{equation} 
then there exists a well-defined (!) linear functional $\tilde{L}$ on $E$ given by
\begin{equation}\label{defiLtilde}
\tilde{L}(f\lceil \cK):=L(f),\quad f\in {\cA},
\end{equation}
and the results on moment functionals on $E$ can be applied to $\tilde{L}$.  There are two important cases where (\ref{constsi}) is satisfied. First, if $f\lceil \cK=0$ implies $f=0$; this happens (for instance) if $\cK$ has a nonempty interior in $\R^n$. Secondly, if $L(f)\geq 0$ for all $f\in \cA$ such that $f\geq 0$ on $\cK$. Then (\ref{constsi}) holds. (Indeed, if $f\lceil \cK=0$, then $\pm f \geq 0 $ on $\cK$,  hence $L(\pm f)\geq 0$, so that $L(f)=0$.) This second case is valid if $L$ is a moment functional which has representing measure supported on $\cK$.
\end{rem}

The following well-known fact will be often  used.

\begin{lem}\label{zerosupp}
Let $f\in C(\cX;\R)$ and $\mu\in M_+(\cX)$. Suppose that $f(x)\geq 0$ for $x\in \cX$ and   $\int f(x)\, d\mu=0$. Then
\[\supp \mu \subseteq \cZ(f)\equiv\{x\in \cX: f(x)=0\}.\]
\end{lem} 
\begin{proof} Let $x_0\in \cX$. Suppose that $x_0\notin\cZ(f)$. Then $f(x_0)>0$. Since $f$ is continuous, there exist an open neighborhood $U$ of $x_0$ and a number $\varepsilon >0$ such that $f(x)\geq \varepsilon $ on $U$. Then 
\[0=\int_\cX f(x)\, d\mu\geq \int_{U}\, f(x) \,d\mu \geq \varepsilon \mu(U)\geq 0,\]
so that  $\mu(U)=0$. Therefore, since $U$ is an open set containing $x_0$, it follows at once from the definition of the support that $x_0\notin \supp\, \mu$.
\end{proof}

A crucial result is the following {\it Richter--Tchakaloff theorem}; it was  proved in full generality by H.\ Richter \cite{richter} and in the compact case by V.\ Tchakaloff  \cite{tchakaloff}.

\begin{prop}\label{hrichtertheorem}
Suppose that  $(\cX,\mu)$ is a measure space and $V$ is a finite-dimensional real   subspace of $\cL^1(\cX,\mu)$. Let $L^\mu$ be the linear functional on $V$ defined by $L^\mu(f)=\int f\, d\mu$, $f\in V$. Then there is a $k$-atomic measure $\nu=\sum_{j=1}^k m_j\delta_{x_j}\in M_+(\cX)$, where $k\leq \dim V $, such that $L^\mu=L^\nu$, that is,
\[\int_\cX f~ d\mu=\int_\cX f~ d\nu\equiv \sum_{j=1}^k m_jf(x_j), \quad f\in V.\]
\end{prop}

An immediate consequence of Proposition \ref{hrichtertheorem} is the following.

\begin{cor}\label{richtercor}
Each moment functional on $E$ of   has a  $k$-atomic representing measure, where $k\leq \dim\, E$.
\end{cor}

For $C\subseteq E $,  a  functional $L$ on $E$ is called  \emph{$C$-positive}\, if $L(f)\geq 0$ for  $f\in C.$ Put
\[E_+:=\{f \in E: f(x)\geq 0\quad \text{for}\quad x\in \cX\}.\]
Obviously, each moment functional is $E_+$-positive.

The dual cone of the cone $E_+$ is the cone  in the dual space $E^*$ of $E$  defined by
\[(E_+)^\wedge= \{L\in E^*: L(f)\geq 0~~\text{for}~~ f\in E_+\}.\]

\begin{dfn}
A linear functional $L$ on $E$ is called \emph{strictly $E_+$-positive} if  
\begin{equation}\label{strictposL}
L(f)>0\quad \text{for all}\quad f\in E_+,~ f\neq 0.
\end{equation}
\end{dfn}

Note that $E_+=\{0\}$ is possible and then every $L$ is strictly $E_+$-positive.

\begin{lem}\label{strictypositivef}
Let $\|\cdot \|$ be a norm   on $E.$ For a linear functional $L$ on $E$ the following are equivalent:
\begin{itemize}
\item[\em (i)] $L$ is strictly $E_+$-positive.

\item[\em (ii)] There exists a number $c>0$ such that
\begin{equation}\label{normLmuesti} 
  L(f)\geq c\|f\| \quad {for}\quad f\in E_+.
\end{equation}

\item[\em (iii)]  $L$ is an interior point of the cone $(E_+)^\wedge$ in $E^*$.
\end{itemize}
\end{lem}
\begin{proof}
If $E_+=\{0\}$, then all assertions are trivially true. So we assume $E_+\neq\{0\}$.

(i)$\to$(ii):  Consider  the set $U_+=\{f\in E_+:\|f\|=1\}$. Since each  point evaluation  $l_x, x\in \cX,$ is continuous on the finite-dimensional normed space $(E,\|\cdot\|)$, $E_+$ is closed in $E$. Hence, $U_+$ is a bounded  closed, hence compact, subset  of $(E,\|\cdot\|)$. Therefore, since the functional $L$ is also  continuous on  $(E,\|\cdot\|),$ the infimum of $L(f)$ on $U_+$ is attained, say at  $f_0\in U_+$. Then $f_0\neq 0$ and $f\in E_+$, so that  $c:=L(f_0)>0$ by (i). Hence $L(f)\geq c$ for $f\in U_+$. By scaling this yields (\ref{normLmuesti}).

(ii)$\to$(iii): We equip $E^*$ with the  dual norm of $\|\cdot \|$. Suppose that $L_1\in E^*$ and $\|L-L_1\|<c$. Then (\ref{normLmuesti}) implies that $L_1(f)\geq 0$ for $f\in E_+$,  that is, $L_1\in (E_+)^\wedge.$ This shows that $L$ is an interior point of the cone $(E_+)^\wedge$.

(iii)$\to$(i): Let $f\in E_+, f\neq 0$. Then there exists $x\in \cX$ such that $f(x)>0$. Since the  point evaluation $l_x$ at $x$ is in $(E_+)^\wedge$ and $L$ is an inner point of $(E_+)^\wedge$, there exists a number $\varepsilon >0$ such that $(L-\varepsilon l_x)\in(E_+)^\wedge$. Hence $L(f)\geq \varepsilon f(x)>0$.
\end{proof}

Let $\ov{\cL}$ denote the closure of the cone $\cL$ in the norm topology of $E^*$.

\begin{lem}\label{e+wedgel}
$(E_+)^\wedge= {\ov{\cL}}.$
\end{lem}
\begin{proof}
Clearly, if $L\in {\cL}$ and $p\in E_+$, then $L(p)\geq 0$. Thus, ${\cL}\subseteq {(E_+)}^{\wedge}$. Therefore, since  $(E_+)^\wedge$ is obviously closed, $\ov{\cL}\subseteq (E_+)^{\wedge}$. 
 
Now we prove the converse inclusion $(E_+)^{\wedge} \subseteq {\ov{\cL}}$\, .
Assume to the contrary that there exists a functional $L_0\in (E_+)^\wedge$ such that $L_0\notin {\ov{\cL}}$\,. Then, by the separation theorem for convex sets  applied to the closed cone ${\ov{\cL}}$\, in $E^*$, there is a linear functional
 $F$   on $E^*$ such that $F(L_0)<0$ and $F(L)\geq 0$ for  $L\in{\cL}$. Since $E$ is finite-dimensional, there is a (unique) element $f\in E$ such that $F(L)=L(f)$ for all $L\in E^*$. Let $x\in \cX$. Then the point evaluation $l_x$ at $x$ is  in ${\cL}$, so that  $F(l_x)=l_x(f)=f(x)\geq 0$. Hence $f\in E_+$. Therefore, since $L_0\in (E_+)^\wedge$, we get $F(L_0)=L_0(f)\geq 0$ which is a contradiction. Thus $(E_+)^\wedge \subseteq {\ov{\cL}}$\,.
\end{proof}

The next proposition  is of similar spirit as a result  proved in \cite{fialkowsur}.

\begin{prop}\label{existencemfstrict}
Each strictly $E_+$-positive linear functional  on $E$ is a moment functional.
\end{prop}
\begin{proof}
Let $L$ be a  strictly $E_+$-positive functional on $E$. Then $L$ is  an inner point of $(E_+)^{\wedge}$ 
by Lemma \ref{strictypositivef} and hence of ${\ov{\cL}}$\, by Lemma  \ref{e+wedgel}. Since the convex set ${\cL}$ and its closure ${\ov{\cL}}$\, have the same inner points, $L$ is also an inner point of $\cL$. In particular, $L$ belongs to ${\cL},$ that is, $L$ is a moment functional.
\end{proof}

\section{The set $\cW(L)$ of atoms}\label{setofatoms}

In this subsection, we assume that the following condition is satisfied:
\begin{equation}\label{possepration}
\text{\textit{For each $x\in \cX$ there exists a function $f_x\in E_+$ such that $f_x(x)>0.$}}
\end{equation} 
The following important concepts appeared already in \cite{matzke} and \cite{schmuedgen}.

\begin{dfn}\label{defVWK}
For a moment  functional $L$ on $E$ we define 
\begin{align} 
\cN_+(L)&=\{ f\in E_+: L(f)=0\, \},\label{defn+s}\\
\cV_+(L)&=\{ x\in \cX: f(x)=0~~{ for~all}~f\in \cN_+(L)\},\label{defv+s}\\ 
\cW(L)&=\{ x\in \cX: \mu(\{x\})>0~~{ for~ some}~\mu \in \cM_L\}.\label{defwl}
\end{align}
\end{dfn}

Thus, $\cW(L)$ is the set of  points $x\in\cX$ which are atoms of some representing measure $\mu$ of $L$. In the important special case $E={\cA}, \cX=\R^n$, $\cN_+(L)$ consists of real polynomials and  $\cV_+(L)$ is a real algebraic set. The sets $\cV_+(L)$ and $\cW(L)$ are fundamental notions in the theory of the truncated moment problem.

\begin{lem}\label{propW_+(L}
Let $L$ be a moment functional on $E$.
\begin{itemize}
\item[\em (i)]~
$\cW(L)\subseteq \cV_+(L).$
 \item[\em (ii)]~ If $L=0$  and $\mu\in \cM_L$, then $\mu=0$.
 \item[\em (iii)]~ The set $\cW(L)$ is not empty if and only if $L \neq 0$.
\end{itemize}
\end{lem}
\begin{proof}
(i): Let $x\in \cW(L)$. By (\ref{defwl}) there is a  measure $\mu\in \cM_L$  such that $\mu(\{x\})>0$. For $f\in \cN_+(L)$, we obtain
$$
0=L(f)=\int f(y)\, d\mu \geq f(x)\, \mu(\{x\})\geq 0.
$$
Since $\mu(\{x\})>0$, it follows that $f(x)=0$. Thus $x\in \cV_+(L)$.

(ii):  Let $x\in \cX$ and let $f_x\in E_+$ be the function from condition (\ref{possepration}). Since $L=0$, we have $L(f_x)=0$. Hence
${\supp}\, \mu\subseteq \cZ(f_x)$ by Lemma \ref{zerosupp}. Therefore,  $\supp\, \mu\subseteq \cap_{x\in \cX} \cZ(f_x).$ Since the latter set is empty  by (\ref{possepration}),  $\mu=0.$

(iii):  By Corollary \ref{richtercor}, $L$ has a finitely atomic  representing measure $\mu$. If $L\neq 0$, then $\mu\neq 0$, so    $\cW(L)$ is not empty. If $L=0$, then $\mu=0$ by (i), so  $\cW(L)$ is empty.
\end{proof}

A natural and important question is whether or not there is equality in Lemma \ref{propW_+(L}(i). The following examples show that this is not true in general, but it holds for the one dimensional truncated moment problem on $[0,1]$.

\begin{exa}\label{exW+neqV+} 
Let $\cX$ be the subspace of $\R^2$ consisting of the three points $ (-1,0)$, $(0,0)$, $(1,0)$ and the two lines $\{(t,1);t\in \R\}$,  $\{(t,-1);t\in \R\}.$ Let $E$ be the restriction to $\cX$ of the polynomials $\R[x_1,x_2]_2$ of degree at most $2$. We easily verify that the restriction map $f\mapsto f\lceil \cX$ on $\R[x_1,x_2]_2$ is injective; for simplicity we  write   $f$ instead of $f\lceil \cX$ for  $f\in \R[x_1,x_2]_2.$

We consider the moment functional $L$ defined by
\begin{equation}\label{defiLE}
L(f)=f(-1,0)+f(1,0), f\in E.
\end{equation}

We show that $\cN_+(L)= \{x_2(b x_2+ c): |c|\leq b,\, b,c\in \R \}.$
It is obvious that these polynomials are in $  \cN_+(L)$. Conversely, let $f\in \cN_+(L)$. Then $f(-1,0)=f(1,0)=0$, so that $f=x_2(ax_1+bx_2+c)+d(1-x_1^2)$, with $a,b,c,d\in \R$. Further, $d=f(0,0)\geq 0.$ From $f(t,\pm 1)\geq 0$ for all $t\in\R$ we conclude that $d=0$ and $|c|\leq b$.
 
The zero set $\cV_+(L)$ of $\cN_+(L)$ is the intersection of $\cX$ with the $x_1$-axis, that is, $\cV_+(L)=\{ (-1,0), (0,0),(1,0)\}$. 
Let $\mu$ be an arbitrary representing measure of $L.$  Then, since $\mu$ is supporting on $\cV_+(L)$,  there are  numbers\,  $\alpha,\beta,\gamma\geq 0$ such that\, $\mu=\alpha \delta_{(-1,0)}+\beta \delta_{(0,0)}+\gamma \delta_{(1,0)}$.  By (\ref{defiLE}), we have $L(x_1)=0=\int x_1~ d\mu=-\alpha+ \gamma$ and $L(x_1^2)=2=\int x_1^2 ~d\mu= \alpha+\gamma,$ which implies that $\alpha=\gamma=1$. Therefore, since $L(1)=2=\int 1~ d\mu= \alpha+\beta+\gamma$, it follows that $\beta=0$. Hence,  $\mu(\{(0,0)\})=0,$ so that   $(0,0)\notin \cW(L)$. Thus, $\cW(L)\neq \cV_+(L)$.  

The preceding proof shows that $L$ has a  unique representing measure.
$ \hfill \circ$
\end{exa}

\begin{exa}
Let  ${\sA}:=\{ 1,x,\dots,x^m\},$ and  $\cX:=[0,1]$. Then we have $\cW(L)=\cV_+(L)$ for each moment functional on $E$. Indeed, if the corresponding moment sequence $s$ is an inner point of the moment cone, then $\cN_+(L)=\{0\}$  and each point of  $[0,1]$ is an atom of a representing measure \cite[Corollary II.3.2]{karlin}. If $s$ is a boundary point of the moment cone, then $s$ has a unique representing measure $\mu$ \cite[Theorem II.2.1]{karlin}. In the first case $\cV_+(L)=\cW(L)=[0,1]$, while  $\cV_+(L)=\cW(L)={\rm supp}\, \mu$ in the second case. $\hfill \circ$
\end{exa}

\begin{lem}\label{propW}
Suppose that $L$ is a  moment functional  on $E$. 
\begin{itemize}
\item[\em (i)] If $\mu\in \cM_L$ and $M\subseteq \cX$ be a Borel  set  containing  $\cW(L)$, then $\mu(\cX\backslash M)=0$. 

\item[\em (ii)] If $\cW(L)$ is finite, there exists a  $\mu\in \cM_L$ such that ${\supp} \mu=\cW(L)$.

\item[\em (iii)] If $\cW(L)$ is infinite, then for any $n\in \N$ there exists a measure $\mu\in \cM_L$ such that $|{\supp} \mu|\geq n$.
\end{itemize}
\end{lem}
\begin{proof} 
The proofs of all three assertions use Proposition  \ref{hrichtertheorem}.

(i): Assume to the contrary that $\mu(\cX\backslash M)>0$ and define linear functionals $L_1$ and $L_2$ on $E$ by 
\[L_1(f)=\int_M f(x)~ d\mu\quad{\rm and}\quad L_2(f)=\int_{\cX\backslash M} f(x)~ d\mu.\]
Applying Proposition  \ref{hrichtertheorem} to the  functionals $L_1$ and $L_2$ and the measure spaces $M$ and $\cX\backslash M$, respectively,  with  measures induced  from $\mu$, we conclude that $L_1$ and $L_2$ have finitely atomic representing measures $\mu_1$ and $\mu_2$ with atoms in $M$ and $\cX\backslash M$, respectively. Since $\mu \in \cM_L$, we have $L=L_1+L_2$ and hence $\tilde{\mu}:=(\mu_1+\mu_2)\in \cM_L.$ From $\mu(\cX\backslash M)>0$ and Lemma \ref{propW_+(L}(ii) it follows that $L_2\neq 0$. Hence $\mu_2\neq 0$. Therefore, if $x_0\in \cX\backslash M$ is an atom of $\mu$, then $\tilde{\mu}(\{x_0\})\geq  \mu_2(\{x_0\})>0$, so that $x_0\in \cW(L)\subseteq M$ which contradicts $x_0\in \cX\backslash M.$ 

(ii): By the definition of $\cW(L)$, for each $x\in\cW(L)$ there is a measure $\mu_x\in\cM_L$ such that $x\in\supp\mu_x$. Then
\[\mu := \frac{1}{|\cW(L)|} \sum_{x\in\cW(L)} \mu_x \in\cM_L\]
and $\cW(L)\subseteq \supp\mu$. (i) implies that $\supp\mu \subseteq \cW(L)$. Thus, ${\supp} \mu=\cW(L)$.

(iii) is proved by a similar reasoning as (ii).
\end{proof}

\begin{thm}\label{strictposLnonunique} 
Each strictly $E_+$-positive linear functional $L$ on $E$ is a moment functional such that
\[\cW(L)=\cX.\]
\end{thm}
\begin{proof}
That $L$ is a moment functional follows  from Proposition \ref{existencemfstrict}.

We fix a norm  $\|\cdot\|$  on $E$. Let $c$ be the corresponding positive number appearing in the inequality (\ref{normLmuesti}) of Lemma \ref{strictypositivef}. 
Suppose that $x\in \cX$. Since the point evaluation  $l_x$ at $x$ is continuous, there is $C_x>0$ such that $|l_x(f)|=|f(x)|\leq C_x\|f\|$ for $f\in E.$ Fix  $\varepsilon$ such that\, $0<\varepsilon C_x<c$. Let $f\in E_+, f\neq 0$. Using (\ref{normLmuesti})  we derive 
\[(L-\varepsilon l_x)( f)\geq c \|f\|- \varepsilon f(x)\geq (c-\varepsilon C_x ) \|f\|> 0.\]
Therefore, by Lemma \ref{strictypositivef},  $L- \varepsilon l_x$\,  is also strictly $E_+$-positive and hence a moment functional by Proposition \ref{existencemfstrict}. 
If $\nu$ is  a representing measure of $L- \varepsilon l_x$, then $\mu:=\nu+ \varepsilon \delta_x$ is a  representing measure of $L$ and $\mu(\{x\})\geq \varepsilon>0.$ Thus, $x\in \cW(L)$.
\end{proof}

\begin{cor}\label{atomcV}
Let $L$ be a moment functional on $E$. Suppose that there exist  a closed subset $\cU$ of $\cX$ and a measure $\mu\in \cM_L$ such that ${\supp}\, \mu \subseteq \cU$  and  the following holds:
If  $f(x)\geq 0$ on $\cU$ and  $L(f)=0$ for some $f\in E$, then $f=0$ on $\cU$.
Then each $x\in \cU$ is atom of some finitely atomic representing measure  of $L$.
\end{cor}
\begin{proof} Being a closed subset of $\cX$, $\cU$ is a locally compact Hausdorff space.
Since ${\supp}\, \mu \subseteq \cU$, there is a well-defined (!) moment functional $\tilde{L}$ on the linear subspace  $\tilde{E}:=E\lceil \cU$ of $C(\cU;\R)$ given by $\tilde{L}(f \lceil \cU)=L(f), f\in E$. In particular, $\tilde{L}$ is $(\tilde{E})_+$-positive on $\tilde{E}$. The condition on $\cU$ implies that $\tilde{L}$ is strictly positive. Hence it follows from Theorem \ref{strictposLnonunique}(ii), applied to $\tilde{L}$ and $\tilde{E}\subseteq C(\cU,\R)$, that $\cW(\tilde{L})=\cU$. Thus each $x\in \cU$ is atom of some representing measure of $\tilde{L}$ and hence of $L$. Corollary \ref{richtercor} implies that this measure can be chosen  finitely atomic.
\end{proof}

\section{Determinacy of moment functionals}\label{Edeterminacy}

\begin{dfn}
A moment functional $L$ on $E$ is called \emph{determinate} if it has a unique representing measure, or equivalently, if the set $\cM_L$ is a singleton.
\end{dfn}

The following theorem  is the main result of this section. It characterizes  determinacy in terms of the size of the set $\cW(L)$.

For $x\in \cX$ we define 
\begin{equation}\label{momentvector}
s_{\sF}(x):=(f_1(x),\dots,f_m(x))^T\in \R^m.
\end{equation} 
Clearly, $s_{\sF}$ the moment vector of the delta measure $\delta_x$.

\begin{thm}\label{dettheorem}
For each moment functional $L$ on $E$ the following are equivalent:
\begin{itemize}
\item[\em (i)]   $L$ is not determinate.

\item[\em (ii)] The set $\{s_{\sF}(x): x\in \cW(L)\}$ is linearly  dependent in $\R^m$.

\item[\em (iii)]  $|\cW(L)|> \dim (E\lceil \cW(L)).$

\item[\em (iv)]  $L$ has a representing measure $\mu$  such that $|{\supp} \, \mu|> {\dim} (E\lceil \cW(L)).$
\end{itemize}
\end{thm}
\begin{proof}
(i)$\to$(iii):
Assume to the contrary that $|\cW(L)|\leq \dim (E\lceil \cW(L))$ and let $\mu_1$ and $\mu_2$ be   representing measures of $L$. Then, since $\dim E$  is finite, so is $\cW(L)$, say $\cW(L)=\{x_1,\dots,x_n\}$ with $n\in \N$. In particular, $\cW(L)$ is a Borel set. Hence, from Lemma \ref{propW}(i), applied to $M=\cW(L),$  we deduce that ${\supp}\, \mu_i\subseteq \cW(L)$ for $i=1,2$, so there are numbers $c_{ij}\geq 0$ for $j=1,\dots,n, i=1,2$, such that 
\[L(f)=\int f(x)\,d \mu_i=\sum_{j=1}^n f(x_j)c_{ij}\quad{\rm for}~~ f\in E.\]
From the assumption $|\cW(L)|\leq \dim (E\lceil \cW(L))$ it follows that there are functions $f_j\in E$ such that $f_j(x_k)=\delta_{jk}.$ Then $L(f_j)=c_{ij}$ for $i=1,2$, so that $c_{1j}=c_{2j}$ for all $j=1,\dots,n$. Hence $\mu_1=\mu_2$, so $L$ is determinate. This contradicts (i).

(iii)$\to$(ii): Since  the cardinality of the set $\{s_{\sF}(x): x\in \cW(L)\}$ exceeds the dimension of $E\lceil \cW(L)$ by (iii), the set must
 be linearly dependent.

(ii)$\to$(i): Since the set $\{s_{\sF}(x): x\in \cW(L)\}$  is linearly dependent, there are pairwise distinct  points $x_1,...,x_k\in \cW(L)$ and real numbers  $c_1,...,c_k$, not all zero, such that $\sum_{i=1}^k c_i s_{\sF}(x_i) = 0.$ Then, since $\{f_1,\dots,f_m\}$ is a basis of $E$, we have
\begin{equation}\label{representLfxi}
\sum_{i=1}^k c_i f(x_i) = 0\quad {\rm for}~~~f\in E.
\end{equation}
We choose for  $x_i\in \cW(L)$ a representing measure $\mu_i$ of $s$ such that $x_i\in\supp\mu_i$. Clearly, $\mu := \frac{1}{k}\sum_{i=1}^k \mu_i$ is a representing measure of $s$ such that $\mu(\{x_i\}) >0$ for all $i$. Let $\varepsilon=\min\ \{\mu(\{x_i\}):i=1,\dots,k\}$.  For each number $c\in (-\varepsilon,\varepsilon)$,
\[\mu_c = \mu + c\cdot \sum_{i=1}^k c_i \delta_{x_i}\]
is a positive (!) measure  which represents $L$ by (\ref{representLfxi}). By the choice of $x_i, c_i$, the signed measure $\sum_i c_i \delta_{x_i}$ is not the zero measure. Therefore, $\mu_c\neq \mu_{c'}$ for $c\neq c'$. This shows that  $L$ is not determinate.

(iii)$\leftrightarrow$(iv): If $\cW(L)$ is finite,  by Lemma \ref{propW}(ii) we can choose $\mu\in \cM_L$ such that ${\supp}\, \mu=\cW(L)$. If  $\cW(L)$ is infinite, Lemma \ref{propW}(iii) implies that there exists $\mu\in \cM_L$ such that  $|{\supp} \, \mu|> {\dim} (E\lceil \cW(L)).$ Thus, the equivalence  (iii)$\leftrightarrow$(iv) is proved in both cases.
\end{proof}

An immediate consequence of Theorem \ref{dettheorem} is the following.

\begin{cor}\label{detcororollary}
If\,  $|\cW(L)|>\dim E$ or if there is a measure $\mu\in \cM_L$ such that $|\supp \, \mu|> \dim E$, then $L$ is not determinate. In particular, $L$ is  not determinate if\, $\cW(L)$ is an infinite set or if\, $L$\, has a representing measure of infinite support.
\end{cor}

\begin{cor}
Suppose that $L$ is a strictly\, $E_+$-positive moment functional on $E$. Then $L$ is determinate if and only if $|\cX|\leq \dim\, E$.
\end{cor}
\begin{proof}
From Theorem \ref{strictposLnonunique} we obtain $\cX=\cW(L)$. Therefore, $\dim E=\dim\, (E\lceil\cW(L))$. Hence the assertion follows from Theorem \ref{dettheorem},(iii)$\leftrightarrow$(i).
\end{proof}

The following simple results contain useful sufficient criteria for determinacy.

\begin{prop}\label{lem:separation1}
Let $s\in\cS$. Suppose that  $\cW(s)= \{x_1,...,x_k\}$, where $k\in \nset$. If for each $j=1,...,k$ there exists $p_j\in E_+$ such that $p_j(x_i) = \delta_{i,j}$, then $s$ is determinate.
\end{prop}
\begin{proof}
Let $\mu$ and $\nu$ be  representing measures of $s$. Then  
$\supp\, \mu\subseteq \cW(s)$ and $\supp\, \nu\subseteq \cW(s)$
by Proposition \ref{propW}(i), so   $\mu = \sum_{i=1}^{k} c_i \delta_{x_i}$ and $\nu = \sum_{i=1}^k d_i \delta_{x_i}.$ Therefore,
\[c_i = \int_\cX p_i~ d\mu = L_s(p_i) = \int_\cX p_i~ d\nu = d_i \qquad\forall i=1,...,k,\]
so that $\mu=\nu$. Hence $s$ is determinate.
\end{proof}

The next proposition contains a sufficient criterion for the existence of such polynomials $p_j$.

\begin{prop}\label{lem:separation2}
Let  $\sB = \{b_1,...,b_k\}$ be a subset of $C(\cX;\rset).$ 
We suppose that ${\sB}^2:= \{b_i b_j:i,j{=}1,\dots,k\}\subseteq E$. Let $s$ be a moment sequence of $E$ such that  $\cW(s)= \{x_1,...,x_l\}$,   $l\leq k$. If the vectors $s_\sB(x_1),...,s_\sB(x_l)\in \rset^k$ are linearly independent, then  there exist functions $p_j\in E_+$ such that $p_j(x_i)=\delta_{ij}$ for  $i,j=1,\dots,l$ and $s$ is determinate.
\end{prop}
\begin{proof}
Recall that $s_{\sB}(x)=(b_1(x),\dots,b_j(x))^T$ for $x\in \cX$. Set
\[M := (s_\sB(x_i)^T)_{i=1,...,l} \quad\text{and}\quad M_j := (s_\sB(x_i)^T)_{i=1,...,j-1,j+1,...,l} \quad\forall j=1,...,l.\]
Since $\{s_\sB(x_1),...,s_\sB(x_l)\}$ is linearly independent, we have
\[\rank M = l \quad\text{and}\quad \rank M_j = l-1 \quad\forall j=1,...,l.\]

Hence, for any $j$ there exists $\tilde{q}_j\in\ker M_j\setminus\ker M$. Then $q_j(x) := \langle \tilde{q}_j,s_\sB(x)\rangle \in {\Lin} \, {\sB}$. We have  $q_j(x_i) = 0$ for  $i\neq j$ and $q_j(x_j)\neq 0$, since otherwise $\tilde{q}_j\in\ker M$. Therefore, $p_j :=q_j(x_j)^{-2}q_j^2 \in E_+$ has the desired properties. 

The determinacy of $s$ follows from Proposition \ref{lem:separation1}.
\end{proof}

In the following example  the Robinson polynomial is used to develop an application of Proposition \ref{lem:separation2}.

\begin{exa}\label{exmp:four}
Let $\sA :=\{ (x,y,z)^\alpha : |\alpha|=6\}$ and  $\sB :=\{ (x,y,z)^\alpha : |\alpha|=3\}$. Then $\sB^2\subseteq E\equiv \cA:=\Lin \, {\sA}$. We  consider the homogeneous polynomials of $\cA$ and $\sB$ acting as continuous functions on the  projective space $\cX=\pset(\rset^2)$. Our aim it to apply Proposition \ref{lem:separation2}. Let
\begin{align*}
\Zset=\{&(1,1,1),(1,1,-1),(1,-1,1),(1,-1,-1),(1,1,0),\\
&(1,-1,0),(1,0,1),(1,0,-1),(0,1,1),(0,1,-1)\}.
\end{align*}
It is known that the  Robinson polynomial
\[R(x,y,z) = x^6 + y^6 + z^6 - x^4(y^2 + z^2) - y^4 (x^2+z^2) - z^4(x^2+y^2) + 3x^2 y^2 z^2\in \]
is non-negative on $\rset^3$,  hence $R\in E_+$,  and that $R$ has the projective zero set $\Zset=\{r_1,...,r_{10}\}$. Then 
\[M := (s_\sB(r_1),...,s_\sB(r_{10}))^T\]
is a full rank $10\times 10$-matrix and for  $i=1,...,10$ the matrix
\[M_i := (s_\sB(r_1),...,s_\sB(r_{i-1}),s_\sB(r_{i+1}),...,s_\sB(r_{10}))^T\]
has rank 9. Hence, by  Proposition \ref{lem:separation2}, there exist polynomials  $p_i\in E_+$ such that $p_i(r_j)=\delta_{i,j}$. Therefore,
\[\Zset(R+p_i) = \{r_1,...,r_{i-1},r_{i+1},...,r_{10}\}.\]
Using polynomials $R+p_{i_1}+...+p_{i_k}$ we find that $\cW_+(s) = \cV_+(s)$ for all moment sequences $s$ with representing measure $\mu$ such that\,  $\supp \mu \subseteq\Zset$.
\end{exa}

In the next example we use the Motzkin polynomial and derive the determinacy from Theorem \ref{dettheorem}.

\begin{exa}\label{exmp:Motzkin}
We consider the Motzkin polynomial
\[M(x,y) = 1 - 3x^2 y^2 + x^2 y^4 + x^4 y^2. \]
Its zero set is $\Zset(M) = \{r_1=(1,1),r_2=(1,-1,),r_3=(-1,1),r_4=(-1,-1)\}$. Let us set $\cX=\R^2$, $\sB := \{1$, $x$, $xy$, $x^3$, $xy^2\}$ and $E=\cA=\Lin\,\sA$, where
\[\sA:=\sB^2=\{ 1, x, x^2, xy, x^3, x^2y, xy^2, x^4, x^2y^2, x^4 y, x^2y^3, x^6, x^4y^2, x^2y^4\}.\]
Then $M\in E_+$. The set $\{s_\sB(r_1),...,s_\sB(r_4)\}$ is linearly dependent, since
\[0 = s_{\sB}(1,1) - s_{\sB}(1,-1) + s_{\sB}(-1,1) - s_{\sB}(-1,1).\]
Hence each  polynomial in $\Lin\, \sB$ vanishing at three roots of $M$ vanishes at the fourth as well. Proposition \ref{lem:separation2} does not  apply. But the set $\{s_\sA(r_1),...,s_\sA(r_4)\}$ is linearly independent. Therefore, the moment sequence of any   measure $\mu = \sum_{i=1}^4 c_i \delta_{r_i}$ is determinate by Theorem \ref{dettheorem},(i)$\leftrightarrow$(ii). 
\end{exa}

\section{Exposed faces of the moment cone}\label{exposedfacesmomentcone}

For $v=(v_1,\dots,v_m)^T\in \R^m$ we abbreviate $$f_v=v_1f_1+\dots+v_mf_m.$$ 
Let $\langle \cdot,\cdot \rangle$ denote the standard scalar product on $\R^m$. Then
\begin{equation}\label{f_v}
f_v(x)=v_1f_1(x)+\dots+v_mf_m(x)=\langle s_\sF(x),v\rangle,\quad x\in \cX.
\end{equation}
Recall that $\sF=\{f_1,\dots,f_m\}$ is a basis of the vector space $E$. Hence $E$ is the set of all functions $f_v$, where $v\in \R^m$. By definition,
\[E_+=\{f_v: v\in \R^m, \langle s_\sF(x),v\rangle \geq 0~~{\rm{for}}~~x\in \cX\}.\]

Further,  we have 
\begin{equation}\label{Lgvt}
L_t(f_v)=\langle v,t\rangle\quad {\rm for}~~~   v,t\in \rset^m.
\end{equation} 
Indeed, since $\R^m ={\cS}-{\cS}$, each vector $t\in\rset^m$ is of the form  $t=\sum_{i=1}^k c_i s_\cF(x_i)$, where  $c_i\in \rset$ and $x_i\in \cX$. Then we compute
\begin{equation}\label{vtL0}
L_t(f_v)=\sum_{i=1}^k c_if_v(x_i)= \sum_{i=1}^k c_i \langle v,s_\sF(x_i)\rangle 
=\langle v,t\rangle
\end{equation}
which proves (\ref{Lgvt}). From (\ref{Lgvt}) it follows that for each linear functional $h$ on $E$ there is a unique vector  $u\in \R^m$ such that 
\begin{equation}\label{h_u}
 h_u(f_v):=h(f_v)=\langle v,u\rangle,\quad v\in \R^m.
\end{equation}
For  $u\in\R^m$ we  define 
\[h_{u}(t):=\langle t,u\rangle, t\in \R^m,\quad \text{and}\quad H_u: = \{x\in\R^m : \langle x,u\rangle = 0\}.\] 
For the set  $\cN_+(s)\equiv \cN_+(L_s)$ defined by (\ref{defn+s}) we have the following  crucial fact.

\begin{lem}\label{N-+(s)}
Let $u\in \R^m$ and $s\in \cS$. Then
$f_u\in \cN_+(s)$ if and only if $h_u(s)=0$ and $h_u(t)\geq 0$ for $t\in {\cS}$. 
\end{lem} 
\begin{proof}
As noted above,  $f_u\in E_+$ if and only if $f_u(x)=\langle s_\sF(x),u\rangle \geq 0$ for all $x\in \cX$. 
Let $t\in \cS$. We can write $t=\sum_i c_i s_\sF(x_i)$ with $x_i\in \cX,$  $c_i\geq 0$ for all $i$. Then
\[L_t(f_u)=\sum\nolimits_i\, c_i \langle s_\sF(x_i),u\rangle =\sum\nolimits_i\, c_i f_u(x_i) =\langle u,t\rangle =h_u(t)\]
by (\ref{h_u}). Hence $f_u\in E_+$ if and only if $h_u(t)\geq 0$ for $t\in {\cS}$. Further,  $L_s(f_u) =0$ if and only if $h_u(s)=0.$   The two latter facts give the assertion.
\end{proof}

Let us recall two basic definition from convex analysis.

\begin{dfn}\label{supportplane}
Let $u\in\R^m, u\neq 0$, and $s\in \cS$. We say that $H_{u}$  is  a \emph{supporting hyperplane} of $\cS$ at the point $s$ if
\[h_u(s)=0 \quad \text{and}~~h_u(t) \geq 0~~\text{for~ all}~~ t\in\cS.\]
The set $H_u\cap \cS$ is called a \emph{proper exposed face} of the cone $\cS$. 
\end{dfn}

Let $s\in \cS$.  Combining Lemma \ref{N-+(s)} and Definition \ref{supportplane} it follows that $H_u$ is a supporting hyperplane of $\cS$ at $s$ if and only if $f_u\in \cN_+(s)$ and $u\neq 0$. Further, $H_u\cap\cS$ is a proper exposed face of $\cS$ if and only if $f_u\in \cN_+(s)$ and $u\neq 0$. All proper exposed faces of $\cS$ are of this form. In the case $u=0$ we have $H_u\cap \cS=\cS$. Note that by definition the sets $\cS$ and $\emptyset$ are also exposed faces of $\cS$.

\begin{prop}\label{N-+(s)11}  Let $s\in \cS$. Then:
\begin{enumerate}
\item[\em (i)] $\cN_+(s)\neq \{0\}$ if and only if $s$ is a boundary point of $\cS$. 
\item[\em (ii)] $\cN_+(s)=\{0\}$ if and only if $s$ is an inner point of $\cS$.
\end{enumerate}
\end{prop}
\begin{proof}
(i): It is well-known from convex analysis that $s$ is a boundary point of $\cS$ if and only if there is a supporting hyperplane $H_u$ of  $\cS$ at $s$. By the preceding the latter holds if and only if $f_u\in \cN_+(s)$ and $f_u\neq 0$. 

(ii) follows from (i) and the obvious fact that $s$ is an inner point of $\cS$ if and only if $s$ is not a boundary point.
\end{proof}

\begin{prop}\label{thm:OneGenerator}
For each  moment sequence $s$  there exists  $p\in \cN_+(s)$ such that
\[\cV_+(s) = \Zset(p):=\{x\in \cX:p(x)=0\}.\]
\end{prop}
\begin{proof}
If $s$ is an inner point of $\cS$, then $\cN_+(s)=\{0\}$ by Proposition \ref{N-+(s)11},  hence $\cV_+(s)=\cX$; so  we can set $p=0$.

Now let $s$ be a boundary point of $\cS$. Let $\cF$ be the set of vector $ u\in \R^m$ such that $h_u(s)=0$ and $ h_u(t)\geq 0$  for all $ t\in {\cS}$. Since $s$ is a boundary point, $\cF$ contains at least one nonzero vector. Then  $\{H_u\cap {\cS}:  u\in \cF, u\neq 0 \}$ is the set of exposed faces of $\cS$. Let $u_1,\dots,u_k$ be a maximal linearly independent subset of $\cF$. Set $u:=u_1+\dots +u_k$.  We show that $p:=f_u$ has the desired properties. Obviously, $u\in \cF$. Hence $f_u\in \cN_+(s)$ by Lemma \ref{N-+(s)} and $H_u\cap {\cS}$ is an exposed face of $\cS$. Thus $\cV_+(s)\subseteq \cZ(p)$ by definition.  Suppose  $x\in \cZ(p)=\cZ(f_u)$. Let $f_v\in \cN_+(s)$. Then $v\in \cF$ by Lemma \ref{N-+(s)}. Hence $v$ is linear combination $v=\sum_i \lambda_i u_i$ of $u_1,\dots,u_k$. Since $f_u(x)=f_{u_1}(x)+\dots+f_{u_k}(x)=0$ and $f_{u_j}\geq 0$, we have $f_{u_i}(x)=0$ for all $i$ and therefore $f_v(x)=\sum_i \lambda_i f_{u_i}(x) =0$. Since $f_v\in \cN_+(s)$ was arbitrary, we have shown that $x\in \cV_+(s)$. 
\end{proof}

Note that the element $p$ in the preceding proposition is not necessarily unique.

For inner points of $\cS$ we have $\cW(s)=\cX$,  hence $\cW(s)=\cV_+(s),$ by Lemma \ref{strictypositivef} and Theorem \ref{strictposLnonunique}. In general, $\cW(s)\neq \cV_+(s)$ as we have seen by  Example \ref{exW+neqV+}.

The next theorem characterizes those boundary points for which $\cV_+(s) = \cW_+(s)$.

\begin{thm}\label{thm:WV+cases}
Let $s$ be a boundary points of $ {\cS}$. Then  $\cV_+(s) = \cW_+(s)$ if and only if  $s$ lies in the relative interior of an exposed face of the moment cone $\cS$.
\end{thm}
\begin{proof} 
By Proposition \ref{thm:OneGenerator}, there exists $p\in \cN_+(s)$ such that $\cZ(p)=\cV_+(s)$. Then $p$ is of the form $p=f_u$ for some $u\in \R^m$ and $H_u$ is a finite-dimensional vector space such that $s_\sF(x)\in H_u$ for  $x\in \cZ(f_u)=\cV_+(s)$. Let us choose  $x_1,...,x_k\in \cZ(f_u)$  such that the vectors $s_\sF(x_1),...,s_\sF(x_k)$ are linearly independent and span  $H_u$. 

First suppose that $\cV_+(s) = \cW(s).$ Then $x_i\in \cW(s)$, so there exists an atomic representing measure $\mu_i$ of $s$ such that $x_i\in\supp \, \mu_i$. Then $\mu := \frac{1}{k}\sum_{i=1}^k \mu_i$ is also a representing measure of $s$ and $x_i\in {\supp}\, \mu$ for all $i$. 

We show that $s$ is an inner point of the exposed face $H_u\cap {\cS}$ of the moment cone $\cS$. Let $v\in H_u$. 
 Since the $s_\sF(x_i)$ are linearly independent and span $H_u$, there are reals $c_1,...,c_k$ such that $v = \sum_{i=1}^k c_i s_\sF(x_i)$. Since the masses of $\delta_{x_i}$ are  positive in $\mu$, there exists a $\varepsilon>0$ such that $s+c\cdot v \in H_u\cap {\cS}$ for all $c\in (-\varepsilon,\varepsilon)$, that is, $s$ is an interior point of the exposed face $H_u\cap {\cS}$.

Conversely,  suppose now that $s$ is an inner point of some exposed face $F$ of $\cS$. Let $x\in \cV_+(s)$.  Then $s_\sF(x)\in  F$. Since $s$ in an inner point, there is a $c>0$ such that $s' = s-c\cdot s(x)\in F$. If $\mu'$ is  representing measure $\mu'$ of $s'$, then $\mu = \mu' + c\cdot \delta_x$ is a representing measure of $s$ and $\mu(\{x\})\geq c>0$, so that $x\in \cW(s)$. Since always  $\cW(s)\subseteq  \cV_+(s)$, we have shown that $ \cW(s)= \cV_+(s)$.
\end{proof}

\section{Set of atoms $\cW(L)$ and  core variety $\cV(L)$}\label{corersetofatoms}

Throughout this section,  $L$ is a moment functional on  $E$ such that $L\neq 0$.

We define inductively subsets $\cN_k(L)$, $k\in \nset,$ of ${\cA}$ and  subsets $\cV_j(L)$, $j\in \nset_0,$ of $\cX$ by $\cV_0(L)=\cX$,
\begin{align*} 
\cN_k(L)&:=\{ p\in \cA: L(p)=0,~~ p(x)\geq 0 ~~{\rm for}~~x\in \cV_{k-1}(L)\, \},~~ k\in \nset,\\ 
\cV_j (L)&:=\{ t\in \cX: p(t)=0~~{\rm for}~p\in \cN_j(L)\},~~ j\in \N.
\end{align*}
If $\cV_k(L)$ is empty for some $k$, we set $\cV_j(L)=\cV_k(L)=\emptyset$ for all $j\geq k, j\in \nset$.

For $k=1$ these notions coincide with those defined  
by (\ref{defn+s}) and (\ref{defv+s}), that is,  $\cN_1(L)=\cN_+(L)\equiv \cN_+(s)$ and  $\cV_1(L)=\cV_+(L)\equiv\cV_+(s),$ where $s$ is the moment sequence of $L.$

The following important concept was defined  and studied by L. Fialkow \cite{fialkoCoreVari}, see also \cite{blekfia},  for arbitrary linear functionals. We will use it only for moment functionals.

\begin{dfn}\label{corevariety} 
The {\em core variety}  $\cV(L)$ of the moment functional $L$ on ${\cA}$ is 
\[\cV(L):=\bigcap_{k=0}^\infty \cV_k(L).\]
\end{dfn}

From the definition it is clear that $\cV_k(L)\subseteq \cV_{k-1} (L)$ for $k\in \nset.$ Further, if $\mu$ is representing measure of $L$, then a repeated application of Lemma \ref{zerosupp} yields
\begin{equation}\label{supprepvl}
{\supp}\, \mu \subseteq \cV (L)\subseteq \cV_{j} (L)\quad \text{for}\quad j\in \nset.
\end{equation}

\begin{prop}\label{termiantes}
There exists  $k\in \nset_0$, $k\leq \dim E$, such that 
\begin{equation}
\cX =\cV_0(L) \supsetneqq \cV_1(L)\supsetneqq ... \supsetneqq \cV_k(L)= \cV_{k+j}(L)=\cV(L),\quad j\in \nset.
\end{equation}
\end{prop}
\begin{proof}
We fix  a representing measure  $\mu$ of $L$. Let $j\in \nset_0$. We    denote by  $E^{(j)}:=E\lceil \cV_j(L)$  the vector space of functions $f\lceil \cV_j(L)$, $f\in E$,  and by $\cL^{(j)}$  the corresponding  cone of moment functionals on $E^{(j)}$. Note that in general $\dim\, E^{(j)}$ is  smaller than  $\dim\, E$. Since ${\rm supp}\, \mu\subseteq \cV_j(L)$ by (\ref{supprepvl}),  $L$  yields a moment functional  $L^{(j)}\in\cL^{(j)}$ given by  $\mu$. Clearly, $E^{(0)}=E$, $\cV_0(L)=\cX$,  $L=L^{(0)}$. By these definitions,  $\cN_{j+1}(L)=\cN_+(L^{(j)})$ and $\cV_{j+1}(L)=\cV_+(L^{(j)})$. From  Proposition \ref{thm:OneGenerator}, applied to the moment sequence of $L^{(j)},$ we conclude that there  exists $p_{j+1}\in E$ such that $p_{j+1}\lceil \cV_j(L)\in \cN_+(L^{(j)})=\cN_{j+1}(L)$   and  
\begin{equation}\label{VJ+j1}
\cV_+(L^{(j)}) =\cV_{j+1}(L)=\cZ(p_{j+1}\lceil \cV_j(s))=\{x\in \cV_j(L): p_{j+1}(x)=0\}.
\end{equation}

First suppose that $L$ is an inner point of $\cL$. Then, by Proposition \ref{N-+(s)11}(ii) we have  $\cN_1(L)=\{0\}$ and hence $\cV_1(L)=\cX$. From the corresponding definitions it follows that $\cN_j(L)=\{0\}$ and  $\cV_j(L)=\cX$ for all $j\in \N$, so the assertion holds with $k=0$.

Now let $L$ be a boundary point of 
$\cL.$ Then $\cN_1(L) \neq \{0\}$ and hence $\cV_1(L) \neq \cX$. Assume that $r\in \nset$ and  $\cV_0(L)\supsetneqq ... \supsetneqq \cV_r(L)$.
We show that  $p_1,\dots,p_{r}$ are linearly independent. Assume  the contrary. Then   $\sum_{j=1}^r \lambda_j p_j=0$, where $\lambda_j\in \rset$, not all zero. Let $n$ be the largest index such that $\lambda_n\neq 0$. Then $p_n(x)=\sum_{j<n} \lambda_j\lambda_n^{-1}p_j$. (The sum is set zero if $n=1$.) Since $\cV_i(L)\subseteq \cV_{j}(L)$ if $j\leq i$ and $p_j$ vanishes on $\cV_j(L)$ by (\ref{VJ+j1}), it follows that $p_n=0$ on $\cV_{n-1}(L)$. Hence $\cV_n(L)\subseteq \cV_{n-1}(L)$ by (\ref{VJ+j1}), a contradiction. 

From the preceding 
two paragraphs 
it follows that there exists  a number $k\in \nset_0$, $k\leq \dim E$, such that $\cV_k(L)= \cV_{k+1}(L)$. Then $\cN_{k+1}(L)=\cN_{k+2}(L)$ and  hence $\cV_{k+1}(L)=\cV_{k+2}(L).$ Proceeding by induction we get $\cV_{k+j}(L)=\cV_k(L)$ for $j\in \nset$, so that $\cV(L)=\cV_k(L)$. 
\end{proof}

\begin{thm}\label{wlequadlvl}
If $L$ is a moment functional on $E$ and $L\neq 0$, then $\cW(L)=\cV(L)$.
\end{thm}
\begin{proof}
From (\ref{supprepvl}) it follows at once that $\cW(L)\subseteq \cV(L)$. 

By Proposition \ref{termiantes}, there exists a number $k\in \nset_0$ such that  (\ref{termiantes}) holds. 
We show that the set $\cU:=\cV(L)$ fulfills the assumptions of Corollary \ref{atomcV}. By  (\ref{supprepvl}), ${\rm supp}\, \mu\subseteq \cV(L)$. Further, if $f\in E$ satisfies $f(x)\geq 0$ on $\cU=V_k(L)$ and $L(f)=0$, then   $f\in \cN_{k+1}(L)$ and hence $f(x)=0$ on $\cV_{k+1}(L)=\cV(L)=\cU.$ Thus Corollary \ref{atomcV} applies and gives the converse inclusion $\cU=\cV(L)\subseteq \cW(L)$.
\end{proof}

\section{Differential structure of the moment cone}\label{differentailstructure}

In this section, we set $\cX=\rset^n$ and assume that $E$ is a {\bf finite-dimensional} linear subspace of $C^1(\rset^n;\rset)$. We will use the differential structure to develop further tools to study the moment cone and moment sequences.

Set $\rset_>:=(0,+\infty)$. Let $C=(c_1,...,c_K)$ and $X=(x_1,...,x_K) $, where $c_j>0$ and $x_j\in \R^n$ for $j=1,\dots,K$, and define a $k$-atomic measure, where $k\leq K$, by
\[\mu_{(C,X)}=\sum_{j=1}^K c_j \delta_{x_j}.\]
Note that $\mu_{(C,X)}$ is not $K$-atomic in general, since  we do not require that the points $x_j$ are pairwise different. By Corollary \ref{richtercor},  each moment sequence has a $k$-atomic representing measure with $k\leq m$. Therefore, if $m\leq K$, the moment sequences of such measures $\mu_{(C,X)}$ exhausts the whole moment cone $\cS$.

We write $(C,X)\in \cM_{K,s}$ if the  $\mu_{(C,X)}$ is a representing measure of  $s\in\cS$.

\begin{dfn}
For  $x\in \R^n$, and $(C,X)\in \R_>^k\times (\R^n)^k$ we define
\begin{equation}
S_{k}(C,X) := \sum_{j=1}^k c_j s_\sF(x_j), \quad \text{where}~~s_\sF(x) := (f_1(x),\dots,f_m(x))^T
\end{equation}
\end{dfn}

Clearly, $S_k$\, is a $C^1$-map of $\R_>^k\times \R^{nk}$ into $\R^m$. Let $DS_k$ denote its total derivative. We write
\begin{equation}
\begin{split}\label{eq:totalderivative}
DS_k (C,X)&= (\partial_{c_1}S_{k},\partial_{x_1^{(1)}}S_{k},...,\partial_{x_1^{(n)}}S_{k},\partial_{c_2}S_{k},...,\partial_{x_k^{(n)}}S_{k})\\
&= (s(x_1), c_1\partial_1 s|_{x=x_1},...,c_1\partial_d s|_{x=x_1},s(x_2),...,c_1 \partial_n s|_{x=x_k}).
\end{split}
\end{equation}
The following is another very simple example for which $\cW(s)\neq \cV_+(s)$.

\begin{exa}
Let $\sF := \{1,x,x^2(x-1)^2\}$, $E={\Lin}\, {\sF}$, and $\cX=\rset $. Set $s: =  s_{\sF}(0) = (1,0,0)^T$. Then
\[DS_1 = (s(0),s'(0)) = \begin{pmatrix} 1 & 0\\ 0 & 1\\ 0 & 0\end{pmatrix}\]
and
\[\ker DS_1^T = \R \cdot v \quad\text{with}\quad v = (0,0,1)^T.\]
Set $p(x) := \langle v, s_{\sF}(x)\rangle = x^2 (x+1)^2$. One verifies that $\cN_+(s)=\rset_\geq \cdot p.$ Hence $\cV_+(s) = \{0,1\}$, but $\cW(s)=\{0\}$. Thus, $\cW(s)\neq \cV_+(s)$.
\end{exa}

\begin{dfn}\label{def:I(s)d(s)}
For  $s\in \cS$ we define the \emph{image} $\Im(s)$, the set $I(s)$,  and the \emph{defect number} $d(s)$ by
\begin{align}
\Im(s) &:= \bigcup_{k\in\nset} \bigcup_{(C,X)\in \cM_{k,s}}\range DS_k(C,X),\\
I(s) &:= \bigcap_{v\in \Im(s)^\bot} \Zset(g_v),\\
d(s) &:= \codim \Im(s) = m - \dim\Im(s).
\end{align}
\end{dfn}

Note that the dimension of $\Im(s)$ is well-defined, since $\Im(s)$ is a linear subspace of $\R^m$ by the following lemma.

\begin{lem}\label{lem:Im(s)}
For each $s\in \cS$ there exist $k\in \N$ and $(C,X)\in \cM_{k,s}$   such that $\Im(s) = \range DS_k(C,X)$. In particular, $\Im(s)$ is a linear subspace of $\R^m$.
\end{lem}
\begin{proof}
Let us prove that $\Im(s)$ is a vector space. Obviously, $\range DS_{k}(C,X)$ is a vector space for any $(C,X)$. Let $v_1, v_2\in\Im(s)$. There are $(C_i,X_i)$, $k_i\in\N$, and $u_i\in\rset^{(m+1)k_i}$ such that
$v_i = DS_{k_i}(C_i,X_i)u_i$.
Then
\begin{align*}
\lambda_1 u_1 + \lambda_2 u_2
&= \lambda_1 DS_{k_1}(C_1,X_1)u_1 + \lambda_2 DS_{k_2}(C_2,X_2)u_2\\
&= \frac{1}{2}(DS_{k_1}(C_1,X_1),DS_{k_2}(C_2,X_2)) \begin{pmatrix}2\lambda_1 u_1\\ 2\lambda_2 u_2\end{pmatrix}\\
&\in\range DS_{k_1+k_2}((C_1/2,C_2/2),(X_1,X_2)) \subseteq \Im(s)
\end{align*}
for any $\lambda_1,\lambda_2\in\rset$.

Let $\{v_1,\dots,v_l\}$ be a basis of the finite-dimensional vector space $\Im(s)$. By the definition of the set $\Im(s)$  for each $i$ there exists  $(C_i,X_i)\in \cM_{k_i,s}$ such that $v_i\in\range DS_{k_i}(C_i,X_i)$. Define
\[k=k_1+\dots,k_l,~ C=(C_1/l,\dots,C_k/l),~ X=(X_1,\dots,X_k).\]
Clearly, $(C,X)\in \cM_{k,s}$. One easily verifies that   $v_i\in\range DS_{k}(C,X)$ for each $i$. Therefore,  $\Im(s) = \mathrm{span}\{v_1,...,v_l\}\subseteq \range DS_{k}(C,X).$ The converse inclusion $\range DS_{k}(C,X)\subseteq \Im(s)$ is trivial. 
\end{proof}

A pair $(C,X)\in \cM_{k,s}$ such that $\Im(s) = \range DS_k(C,X)$ is called a \emph{representing measure of $\Im(s)$}.

Let $\tilde{g}_1,...,\tilde{g}_{d(s)}\in\R^m$ be a basis of $\Im(s)^\perp$, where $\Im(s)^\perp$ denotes  the orthogonal complement of $\Im(s)$ with respect to the standard scalar product of $\R^m$. Then
\[g_i(\,\cdot\,) := \langle s_\sF(\,\cdot\,),\tilde{g}_i\rangle\]
are functions of $E$ and 
\[I(s) := \bigcap_{i=1}^{d(s)} \Zset(g_i).\]

\begin{lem}\label{lem:W(s)I(s)V(s)}
\begin{enumerate}
\item[\em (i)]  ${\Lin}\, \{\tilde{g}_1,...,\tilde{g}_{d(s)}\} = {\ker} DS_{k}(C,X)^T$ for each representing measure $(C,X)$ of $\Im(s)$.

\item[\em (ii)]  ${\Lin}\,\{g_1,...,g_{d(s)}\} = \{p\in E: p_{|\cW(s)}=0 \;\text{and}~ \partial_j p_{|\cW(s)}=0\; \forall j=1,...,n\}$.

\item[\em (iii)]  $\cW(s)\subseteq I(s)\subseteq \cV_+(s)$.

\item[\em (iv)]  $d(s) = 1 \;\Rightarrow\; I(s) = \cV_+(s)$.
\end{enumerate}
\end{lem}
\begin{proof}
(i) follows at once from the corresponding definitions.

(ii): For ``$\subseteq$'' we have $\tilde{g}_i \perp s(x)$ and $\tilde{g}_i \perp \partial_j s(x)$ for all $x\in \cW(s)$, $j=1,...,n$, and $i=1,...,d(s)$, i.e.,
\[g_i(x) = \langle\tilde{g}_i,s_\sF(x)\rangle = 0\]
and
\[\partial_j g_i(x) = \langle\tilde{g}_i,\partial_j s_{\sF}(x)\rangle = 0\]
for all $ x\in \cW(s)$, $j=1,...,n$, and $i=1,...,d(s)$.

For ``$\supseteq$'' let $p(\,\cdot\,) = \langle\tilde{p},s_\sF(\,\cdot\,)\rangle\in E$ such that $p_{|\cW(s)}=0$ and $\partial_j p_{|\cW(s)} = 0$. Then $\tilde{p}\in\Im(s)^\perp=\mathrm{span}\,\{\tilde{g}_1,...,\tilde{g}_{d(s)}\}$.

(iii): First we prove that $\cW(s)\subseteq I(s)$. Let $x\in \cW(s)$, i.e., there is a  representing measure $(C,X)$ with $X=(x,x_2,...,x_k)$. Then
\begin{align*}
s_\sF(x)\in\range DS_{k}(C,X) \;&\Rightarrow\; s_\sF(x)\in\Im(s)\\
&\Rightarrow\; s_\sF(x)\perp \tilde{g}_i \;\forall i=1,...,d(s)\\
&\Rightarrow\; g_i(x) = \langle s_\sF(x),\tilde{g}_i\rangle = 0 \;\forall i=1,...,d(s)\\
&\Rightarrow\; x\in I(s).
\end{align*}

Now we prove that $I(s)\subseteq \cV_+(s)$. From (ii) and the inclusion 
\[\cN_+(s) \subseteq \{ p\in E \,|\, p_{|\cW(s)}=0 \;\text{and}\; \partial_i p_{|\cW(s)}=0\}\]
it follows that
\[I(s) =  \bigcap_{g\in\Im(s)^\perp} \Zset(g) \subseteq \bigcap_{p\in \cN_+(s)} \Zset(p) = \cV_+(s).\]

(iv): Since $d(s)=1>0$, $s$ is not an inner point of the moment cone. Hence there exists $p\in\cN_+(s), p\neq 0.$ Then $p\in\Lin \{g_1\} = \R\cdot g_1$ by (ii), i.e., $p=c\cdot g_1$ for some $c\neq 0$ and $I(s)=\cZ(g_1)=\cV(s)$.
\end{proof}

The following  examples show  that both inclusions in  (iii) can be   strict. In fact,  $\cW(s) \subsetneq I(s) = \cV_+(s)$ in Example \ref{exmp:one} and $\cW(s) = I(s) \subsetneq \cV_+(s)$ in Example \ref{exmp:two}.

\begin{exa}\label{exmp:one}
Let $\sA:= \{1,x^2,x^4,x^5,x^6,x^7,x^8\}$ and $a,b\in\R\setminus\{0\}$ s.t.\ $|a|\neq |b|$. Set $\mu = c_1 \delta_{-a} + c_2 \delta_{a} + c_3 \delta_b$ with $c_i>0$ for $i=1,2,3$ and let $s$ be the moment sequence of $\mu$. Then we have $\ker (DS_3)^T = \R \cdot v$, where 
\[ v = (a^4 b^4,-2 (a^4 b^2+a^2 b^4),a^4+4 a^2 b^2+b^4,0,-2 \left(a^2+b^2\right),0,1)^T,\]
\[p(x) := \langle v, s_\sA(x)\rangle = (x^2-a^2)^2  (x^2-b^2)^2.\]
Hence $\cV_+(s) = \{a,-a,b,-b\}.$ We show that $\cW(s) = \{a,-a,b\}.$ Clearly,  $\{a,-a,b\}\subseteq \cW(s)\subseteq \cV_+(s)$. Assume to the contrary that $\cW(s)=\{a,-a,b,-b\}$. Then $s$ has a representing measure of the form $\mu^* = c_1^* \delta_{-a} + c_2^* \delta_{a} + c_3^* \delta_{-b} + c_4^* \delta_b$ and we obtain
\begin{align*}
s &= c_1 s_\sA(-a) + c_2 s_\sA(a) + c_3 s_\sA(b) = c_1^* s_\sA(-a) + c_2^* s_\sA(a) + c_3^* s_\sA(b) + c_4^* s_\sA(-b)\\
0 & = (c_1^*-c_1) s_\sA(-a) + (c_2^*-c_2) s_\sA(a) + (c_3^*-c_3) s_\sA(b) + c_4^* s_\sA(-b).
\end{align*}
This implies that $c_4 = 0$ and $c_i = c_i^*$ for $i=1,2,3.$ Thus $\mu=\mu^*$ and we have
\[\cW(s) \subsetneq I(s) = \cV_+(s).\]
\end{exa}

\begin{exa}\label{exmp:two}
Let $\cX=\rset$,  
\[{\sA} :=  \{ x^{3\alpha+5\beta} \,|\, \alpha,\beta = 0,1,2\} = \{1, x^3, x^5, x^6, x^8, x^{10}, x^{11}, x^{12}, x^{16}\}\]
and let $s$ be the moment sequence of  $\mu := c_1 \delta_{-1} + c_2 \delta_0 + c_3 \delta_1 + c_4 \delta_2.$ Then
\[\ker DS_4^T = v_1 \R + v_2 \R,\]
where
\begin{align*}
v_1 &= (0, 12864, -17152, -14580, 29163, -14584, 4288, 0, 1)^T,\\
v_2 &= (0, 192, -256, -220, 441, -222, 64, 1, 0)^T,
\end{align*}
so that
\begin{align*}
p_1(x) &= \langle v_1, s_\sA(x) \rangle\\
& = (x-2)^2 (x-1)^2 x^3 (x+1)^2 (x^7+4 x^6+14 x^5+40 x^4+107 x^3+4556 x^2\\
&\ \qquad\qquad\qquad\qquad\qquad\qquad\;\; +3216 x+3216),\\
p_2(x) &= \langle v_2, s_\sA(x) \rangle\\
&= (x-2)^2 (x-1)^2 x^3 (x+1)^2 \left(x^3+68 x^2+48 x+48\right).
\end{align*}
But neither $p_1$ nor $p_2$ is non-negative. It it not difficult to verify that
\[p(x): =  p_1(x) - 67 p_2(x)
= (x-2)^2 (x-1)^2 x^6 (x+1)^2 (x+2)^2 \left(x^2+10\right)\]
is, up to a constant factor, the only non-negative element of  ${\cA}$  such that $L_s(p) = 0$. Therefore, $\cN_+(s)=\R_+ {\cdot}\ p$ and $\cV_+(s) = \{-2,-1,0,1,2\}$. Using that the vectors $s_\sA(-2),s_\sA(-1),s_\sA(0),s_\sA(1), s_\sA(2)$ are linearly independent, a similar reasoning  as in the preceding example shows  $\cW(s) = \{-1,0,1,2\}$. Thus, 
\[\cW(s) = I(s) \subsetneq \cV_+(s).\]
\end{exa}

The next theorem is the main result of this section. It collects various characterizations of   inner points of the moment cone.

\begin{thm}\label{thm:W(s)V(s)innermomentsequence}
For $s\in {\cS}$ the following are equivalent:
\begin{enumerate}
\item[\em (i)] $s$ is an inner point of the moment cone $\cS$.

\item[\em (ii)]\, $\cN_+(s)= \{0\}$. 

\item[\em (iii)]\,  $\cW(s)=\R^n$.

\item[\em (iv)]\, $I(s)=\R^n$.

\item[\em (v)]\, $\cV_+(s)=\R^n$.

\item[\em (vi)] $d(s)=0$.

\item[\em (vii)]\,  $\Im(s)=\R^m$.
\end{enumerate}
\end{thm}
\begin{proof}
(i)$\leftrightarrow$(ii) follows from  Lemma \ref{N-+(s)}(iii).

(vi)$\leftrightarrow$(vii) is Definition  \ref{def:I(s)d(s)}.

(i)$\rightarrow$(iii): Let $x\in\R^n$. Since $s$ is an inner point, so is $s' := s -\varepsilon s(x)$ for some $\varepsilon>0$. If $\mu'$ is a representing measure of $s'$, then $\mu := \mu' + \varepsilon \delta_x$ represents $s$ and $\mu(\{x\})\geq \varepsilon >0$, so that  $x\in \cW(s)$.

(iii)$\rightarrow$(iv)$\rightarrow$(v) follows from Lemma \ref{lem:W(s)I(s)V(s)}(iii).

(v)$\rightarrow$(ii): Let $p\in \cN_+(s)$. Then, since $\rset^n=\cV_+(s)\subseteq \Zset(p)$ by (v), $p=0$.

(vii)$\to$(i): By Lemma  \ref{lem:Im(s)} there exists a representing measure $(C,X)$ of $s$ such that $DS_{k}(C,X)$ has full rank. But then an open neighborhood of $(C,X)$ is mapped onto an open neighborhood of $s$, i.e., $s$ is an inner point.

(iii)$\to$(vii): Since $\cW(s)=\R^n$ and the functions $f_1,...,f_m$ are linearly independent, there are points $x_1,...,x_m\in\R^n=\cW(s)$ such that the vectors $s_\sF(x_1),...,s_\sF(x_n)$ are linearly independent. Then for any $x_i$ there is an atomic measure $\mu_i$ such that $x_i\in\supp\mu_i$. Setting $\mu := \frac{1}{m}\sum_{i=1}^m \mu_i$,  $\mu$ is a representing measure of $s$ and $\range DS_{k_1+...+k_m}(\mu)$ is $m$-dimensional, so that $\R^m\subseteq\Im(s)\subseteq\R^m$.
\end{proof}

Since $s\in {\cS}$ is a boundary point if and only if it is not inner, the following corollary restates the preceding theorem.

\begin{cor}\label{cor:BoundaryPoints}
For each $\in {\cS}$ following statements are equivalent:
\begin{enumerate}
\item[\em (i)] $s$ is a boundary point of the moment cone.

\item[\em (ii)]  $\cN_+(s)\neq \{0\}.$

\item[\em (iii)]  $\cW(s)\subsetneq\R^n$.

\item[\em (iv)]  $I(s)\subsetneq\R^n$.

\item[\em (v)]  $\cV_+(s)\subsetneq\R^n$.

\item[\em (vi)]  $d(s)>0$.

\item[\em (vii)]  $\Im(s)\subsetneq\R^m$.
\end{enumerate}
\end{cor}

\begin{cor}
Suppose that $s$ is a boundary point of $\cS$  with representing measure $(C,X)$. If\, $\codim\range DS(C,X) = 1$, then $\Im(s) = \range DS(C,X)$.
\end{cor}
\begin{proof}
Since $1 = \codim\range DS(C,X) \geq d(s) \geq 1$, it follows that $(C,X)$ is a representing measure of $\Im(s)$.
\end{proof}

The following proposition collects a number of useful properties of the set $\cM_{k,s}\equiv S_k^{-1}(s)$ of at most $k$-atomic representing measures of $s$.

\begin{prop}
Suppose that $n\in \nset$ and $E \subset C^r(\R^n,\R)$, $r\geq 0$. Let  $s\in\cS$.
\begin{enumerate}
\item[\em (i)]  The set $S_{k}^{-1}(s)$ of $k$-atomic representing measures $(C,X)$ of $s$ is closed.

\item[\em (ii)]  $S_{k+1}^{-1}(s)_{|c_{k+1}=0} = S_{k}^{-1}(s)\times\{0\}\times\R^n$.

\item[\em (iii)]  Suppose that $(C,X)$ is an at most $k$-atomic representing measure of $s$ and $DS_{k}(C,X)$ has full rank. In a neighborhood of $(C,X)$,  $S_{k}^{-1}(s)$  is a $C^r$-manifold of dimension $k(n+1)-m$ in $\R^{k(n+1)}$. The tangent space $T_{(C,X)}S_{k}^{-1}(s)$ at $(C,X)$ is
\begin{equation}\label{eq:tangentspace}
T_{(C,X)}S_{k}^{-1}(s) = \ker DS_{k}(C,X).
\end{equation}

\item[\em (iv)] If $s$ is regular, then $S_{k}^{-1}(s)$ is a $C^r$-manifold and (\ref{eq:tangentspace}) holds at any representing measure $(C,X)$.

\item[\em (v)] $\cW(s) =\{x\in\R^n : (c_1,...,c_k;x,x_2,...,x_k)\in S_{k}^{-1}(s),\,  c_1 > 0,\, \text{for some}\, k\geq 1\}$.
\end{enumerate}
\end{prop}
\begin{proof}
The continuity of the map $S_{k}$ gives (i). (ii) is obvious.

(iii) and (iv) are straightforward applications of the implicit function theorem.

(v) follows easily from the definitions of $\cW(s)$ and $S_{k}$.
\end{proof}

\noindent
{\bf Acknowledgement:} The authors  thank Profs. G. Blekherman, L. Fialkow and L. Tuncel for valuable discussions at the Oberwolfach meeting, March 2017.

\bibliographystyle{amsalpha}

\end{document}